\documentclass{amsart}
\usepackage[utf8]{inputenc}

\usepackage{hyperref}
\usepackage{epsfig}
\usepackage{appendix}
\usepackage{amsfonts,amssymb,amsmath,amsthm,amscd}
\usepackage{latexsym}
\usepackage{mathtools}
\usepackage{graphicx}
\usepackage[all, cmtip]{xy}

\newtheorem{innercustomgeneric}{\customgenericname}
\providecommand{\customgenericname}{}
\newcommand{\newcustomtheorem}[2]{%
  \newenvironment{#1}[1]
  {%
   \renewcommand\customgenericname{#2}%
   \renewcommand\theinnercustomgeneric{##1}%
   \innercustomgeneric
  }
  {\endinnercustomgeneric}
}

\newcustomtheorem{thm}{Theorem}
\newcustomtheorem{cor}{Corollary}

\theoremstyle{plain}
		\newtheorem{theorem}{Theorem}[section]
		
		\newtheorem{def-thm}{Definition-Theorem}[section]
		\newtheorem{lemma}[theorem]{Lemma}
		\newtheorem{corollary}[theorem]{Corollary}
		\newtheorem{proposition}[theorem]{Proposition}

		\newtheorem{problem}[theorem]{Problem}

		\newtheorem*{theorem*}{Theorem}

\theoremstyle{definition}
		\newtheorem{definition}[theorem]{Definition}

\theoremstyle{remark}

\newcommand{\A}{{\mathbb{A}}}

\renewcommand{\P}{{\mathbb{P}}}
\newcommand{\Q}{{\mathbb{Q}}}

\newcommand{\V}{{\mathbb{V}}}

\newcommand{\Z}{{\mathbb{Z}}}

\newcommand{\Lcal}{{\mathcal{L}}}

\newcommand{\Ocal}{{\mathcal{O}}}

\newcommand{\mfrak}{{\mathfrak{m}}}

\newcommand{\id}{{\textup{id}}}

\DeclareMathOperator{\Aut}{Aut}

\DeclareMathOperator{\End}{End}

\DeclareMathOperator{\Fin}{Fin}

\DeclareMathOperator{\Hom}{Hom}

\DeclareMathOperator{\sheafhom}{\mathcal{H}\kern -.5pt \emph{om}}

\DeclareMathOperator{\Mat}{Mat}

\DeclareMathOperator{\Pic}{Pic}

\DeclareMathOperator{\PGL}{PGL}

\DeclareMathOperator{\SL}{SL}

\DeclareMathOperator{\Span}{Span}

\DeclareMathOperator{\Tor}{Tor}

\DeclareMathOperator{\Preper}{Preper}

\usepackage{booktabs}

\title{On finite categories of algebraic varieties}
\author{Junho Peter Whang}
\address{Department of Mathematical Sciences and RIM,
Seoul National University}
\email{jwhang@snu.ac.kr}
\date{\today}

\begin{document}

\begin{abstract}
We prove that the finiteness of a finitely generated category of irreducible algebraic varieties over a field of characteristic zero is decidable. We also obtain a Burnside finiteness criterion for such a category, with applications to algebraic dynamical systems of several maps.
\end{abstract}

\maketitle


\section{Introduction}\label{sect:1}
\subsection{\unskip} Jacob \cite{jacob} 
showed that the finiteness of a finitely generated monoid of matrices over a field is decidable. This paper provides a nonlinear generalization of this result for finitely generated categories of irreducible algebraic varieties, over fields of characteristic zero. In this paper, by an algebraic variety we mean a reduced separated scheme of finite type over a field. Let us define a \emph{system} in a category $C$ to be a quiver (i.e., directed multigraph) whose vertices are objects in $C$ and whose arrows are morphisms in $C$ between the vertices. In other words, a system in $C$ specifies the generators of a subcategory of $C$.

\begin{theorem}
\label{mainthm0}
Let $k$ be a field of characteristic zero. There exists an algorithm to determine, given an explicit finite system of irreducible algebraic varieties over $k$, whether or not the category it generates is finite.
\end{theorem}

We make precise the notion of an explicitly given system of varieties in Section \ref{sect:3.1}. We prove Theorem \ref{mainthm0} by induction on the complexity of the system, using two ingredients: an effective form of nonlinear Selberg's lemma due to Bass-Lubotzky \cite{bl}, and the observation that dominant endormophisms of finite order on integral schemes are automorphisms (Lemma \ref{finite-order}). A similar argument yields a solution to the strong Burnside problem for categories of varieties in characteristic zero, generalizing (in characteristic zero) the work of McNaughton-Zalcstein \cite{mz} on matrix monoids. Let us say that a category is \emph{torsion} if every endomorphism of every object generates a cyclic monoid of finite order under composition.

\begin{theorem}
\label{mainthm1}
Let $k$ be a field of characteristic zero. Let $C$ be a finitely generated subcategory of the category of irreducible algebraic varieties over $k$. Then $C$ is finite if and only if it is torsion.
\end{theorem}

Over general fields, one can deduce a weaker finiteness criterion for categories of varieties, establishing the analogue of the bounded Burnside problem. This relies on Zelmanov's resolution of the restricted Burnside problem \cite{z1, z2}. One source of motivation for our work is the study of finite orbits in dynamics of several maps on algebraic varieties. For instance, we prove the following.

\begin{corollary}
\label{cor1}
Let $k$ be a field. Let $M$ be a finitely generated monoid acting on an algebraic variety $V/k$. Let $x\in V(k)$. Then $M\cdot x$ is finite if and only if
    $$\sup_{N} |N\cdot x|<\infty$$
where $N$ runs over all $2$-generated submonoids of $M$.
\end{corollary}

Given a set $X$ and a monoid $M$ of endomorphisms of $X$, let us say that a point $x\in X$ is \emph{$M$-periodic} if the $M$-orbit $M\cdot x$ is finite and $M$ permutes the elements of $M\cdot x$. In the case where $k$ is a field of characteristic zero and $M$ is a finitely generated monoid acting on a variety $V/k$, one can combine a refined form of Corollary \ref{cor1} with \cite[Theorem 1.2]{whang} to obtain the following.

\begin{theorem}
\label{finorb}
Let $k$ be a field of characteristic zero.
Let $M$ be a finitely generated monoid acting on an algebraic variety $V/k$. Let $x\in V(k)$. Then the following are equivalent:
\begin{enumerate}
    \item [(a)] $x$ is $M$-periodic.
    \item [(b)] $x$ is $N$-periodic for every $2$-generated submonoid $N\leq M$.
\end{enumerate}
If moreover $M$ is a group, then the above are equivalent to:
\begin{enumerate}
    \item [(c)] $x$ is $\langle f\rangle$-periodic for every $f\in M$.
\end{enumerate}
\end{theorem}

This paper is organized as follows. Section \ref{sect:2} collects the necessary background, including results of Bass-Lubotzky \cite{bl} and Zelmanov \cite{z1, z2}. In Section \ref{sect:3}, we prove Theorems \ref{mainthm0} and \ref{mainthm1}. In Section \ref{sect:4}, we consider dynamical corollaries of our main results, and in particular prove Corollary \ref{cor1} and Theorem \ref{finorb}.

\subsection{Acknowledgments}
I thank Abhishek Oswal for helpful conversations. In particular, the proof of Theorem \ref{mainthm0} was inspired by an unpublished collaborative work on finite matrix monoids. This work was supported by the Samsung Science and Technology Foundation under Project Number SSTF-BA2201-03.

\section{Background}\label{sect:2}
\subsection{Notations} Let us set up the notations and terminology that will be used throughout the paper. A quiver is a directed multigraph. Given a quiver $Q$, we write $Q_0$ for the class of its vertices and $Q_1$ for the class of its arrows. Let $s,t:Q_1\rightrightarrows Q_0$ denote the source and target maps of $Q$. We shall say that a quiver is \emph{small} if $Q_0$ and $Q_1$ are sets. We shall view a category as a quiver along with a composition law on its arrows that satisfies the usual axioms.

Let $C$ be a category. If $X,Y\in C_0$, we write
$\Hom_{C}(X,Y)$
for the set of morphisms in $C$ from $X$ to $Y$. We write $\End_C(X)=\Hom_{C}(X,X)$. Given $X\in C_0$, we write $C_X$ for the full subcategory of $C$ with a single object $X$. Thus, $(C_X)_1=\End_C(X)$. In this paper, a \emph{monoid} is a small category whose object set is a singleton. A monoid is \emph{cyclic} if it is generated by a single endomorpism. A \emph{groupoid} is a small category whose morphisms are all invertible.

\begin{definition}
Let $C$ be a small category.
    \begin{enumerate}
        \item The \emph{order} of $C$ is $|C|=| C_1|$.
        \item We say $C$ is \emph{finite} if it has finite order. 
        \item We say $C$ is \emph{torsion} if every cyclic submonoid of $C$ is finite.
        \item We say $C$ is \emph{$n$-torsion} if every cyclic submonoid of $C$ has order $\leq n$.
    \end{enumerate}
\end{definition}

Forgetting the composition law on a category $C$, we obtain a quiver (i.e.~directed multigraph) whose collection of vertices is $C_0$, whose collection of arrows is $C_1$, and whose source and target maps are $s$ and $t$.

\begin{definition}
A \emph{system} in a category $C$ to be a small subquiver of the quiver underlying $C$. Given a system $S$ in $C$, let $\langle S\rangle$ denote the subcategory of $C$ generated by $S$, i.e.~smallest subcategory $C'$ of $C$ such that $S_0=C'_0$ and $S_1\subseteq C'_1$. We say that a category $C$ is \emph{finitely generated} if there is a finite quiver $S$ in $C$ (i.e.~with $|S_0|$ and $|S_1|$ finite) such that $C=\langle S\rangle$.
\end{definition}




\begin{definition}
Let $S$ be a quiver. A \emph{path} in $S$ from $v\in S_0$ to $w\in S_0$ is a sequence $f_1,\dots,f_k\in S_1$ of arrows such that $s(f_1)=v$, $t(f_k)=w$, and $t(f_i)=s(f_{i+1})$ for $i=1,\dots,k-1$. Two vertices $v,w\in S_0$ are \emph{path-equivalent} if there is a path in $S$ from $v$ to $w$ and there is a path in $S$ from $w$ to $v$. We denote by $S^\circ$ the quiver obtained from $S$ by deleting the arrows between vertices that are not path-equivalent. A system $S$ is \emph{path-connected} if $S=S^\circ$. A \emph{path-component} of $S$ is a maximal subquiver of $S$ that is path-connected.
\end{definition}

\begin{lemma}
\label{connected}
A finitely generated category $C$ is finite if and only if $C^\circ$ is finite.
\end{lemma}

\begin{proof}
If $C$ is finite, then obviously $C^\circ$ is finite. Suppose conversely that $C^\circ$ is finite. Let $N$ be the number of path-components of $C$. Let $S$ be a finite system of generators for $C$. Every morphism of $f$ can be written in the form
$$f=g_kh_{k}g_{k-1}h_{k-1}\dots g_{2}h_{1}g_{0}$$
for some $k\leq N$, where $g_i\in c(C)_1$ and $h_i\in S_1\setminus c(C)_1$ for each $i=0,\dots, k$. It follows that $C_1$ is finite, so $C$ is finite.
\end{proof}

\subsection{Bass-Lubotzky}\label{sect:3.1}
We recall a theorem of Bass-Lubotzky \cite[Corollary (1.2)]{bl}.

\begin{theorem}
\label{bl0}
    Let $k$ be an arbitrary ring. Let $G$ be a finitely generated group of automorphisms of a scheme $V$ of finite presentation over $k$.
    \begin{enumerate}
        \item $G$ is residually finite.
        \item If $V$ is flat over $\Z$, then $G$ is virtually torsionfree.
    \end{enumerate}
\end{theorem}

We shall also need an effective form of the second part of Theorem \ref{bl0}, whose formulation we recall as follows. Let $k$ be a finitely generated subring of $\bar\Q$. Let $Z$ be a scheme flat of finite type over $k$. Given a finite set $X$ of closed points of $Z$, let
    $$A_{X}=\prod_{x\in X}\Ocal_{Z,x}\quad\text{and}\quad J_{X}=\prod_{x\in X}\mfrak_{Z,x}$$
where $\Ocal_{Z,x}$ denotes the local ring of $Z$ at $x$ and $\mfrak_{Z,x}$ is the maximal ideal of $\Ocal_{Z,x}$, with residue field $\kappa(x)=\Ocal_{Z,x}/\mfrak_{Z,x}$. We shall say that $X$ has \emph{residue characteristic $p$} if $\kappa(x)=p$ for every $x\in X$. Following \cite{bl}, we shall say that $X$ is \emph{effective} if there is a finite affine open covering $(U_i)_{i\in I}$ of $Z$ such that the natural morphism
    $\Ocal_Z(U_i)\to\prod_{x\in X\cap U_i}\Ocal_{Z,x}$
    is injective for every $i\in I$. The following holds.

\begin{proposition}
\label{prop1}
    Suppose that $X$ is a finite effective set of closed points in $Z$ with residue characteristic $p$. Then the order of any torsion element in $$\Gamma_X=\ker(\Aut(Z/k)\to\Aut Z(A_{X}/J_{X}^2))$$ is a power of $p$. In particular, if $X$ and $X'$ are finite effective sets in $Z$ having distinct residue characteristics $p\neq p'$, then $\Gamma_X\cap\Gamma_{X'}$ is a normal torsionfree subgroup of finite index in $\Aut(Z/k)$.
\end{proposition}

\begin{proof}
    The proof is given in \cite[pp.4-5]{bl}. See also \cite[Section 2.1]{whang} for a summary.
\end{proof}

\subsection{Burnside problem}
We refer to \cite{reference} for a summary of the history of the Burnside problem. Here, we recall their formal statements of its variants.
\begin{problem}
    [Strong Burnside's problem]
    Let $G$ be a finitely generated group all of whose elements are torsion. Is $G$ finite?
\end{problem}
\begin{problem}
    [Bounded Burnside's problem]
    Let $G$ be a finitely generated group all of finite exponent. Is $G$ finite?
\end{problem}
\begin{problem}
    [Restricted Burnside's problem]
    Are there only finitely many finite groups with given number of generators and given exponent?
\end{problem}
While both the strong Burnside's problem and bounded Burnside's problem have negative answers in general (in the strong case by work of Golod-Shafarevich \cite{gs}, and in the bounded case by work of Adian and Novikov \cite{an}), for linear groups they admit affirmative answers (due to Schur \cite{schur} and Burnside \cite{burnside}, respectively). The restricted Burnside problem was solved affirmatively by Zelmanov \cite{z1,z2}. An immediate corollary of his work is the following characterization of finite groups.

\begin{theorem}
    \label{zelmanov}
    A group $G$ is finite if and only if it is finitely generated, residually finite, and of finite exponent.
\end{theorem}

In the meanwhile, analogues of Burnside's problems for other algebraic structures such as semigroups have been studied. In \cite{mz}, McNaughton-Zalcstein \cite{mz} established the analogue of the strong Burnside's problem for semigroups (or monoids) of matrices over arbitrary fields. Theorem \ref{mainthm1} serves to generalize this result, in characteristic zero, to categories of algebraic varieties.

\section{Proofs of Theorems \ref{mainthm0} and \ref{mainthm1}}\label{sect:3}



\subsection{Explicitly given systems of varieties}
Here, we make precise our notion of an explicitly given finite systems of algebraic varieties, generally following the spirit of the paragraph after \cite[Theorem 1.2]{whang}.

\begin{enumerate}
    \item A finitely presented ring $k$ is explicitly given if it is given as a quotient of a polynomial ring with coefficients in $\Z$, and a finite set of generators for the kernel of the quotient map is specified. A ring homomorphism between explicitly given finitely presented rings is explicitly given if the images of the generators (given by the explicit finite presentation) of the domain ring are specified. In what follows, let $k$ be an explicitly given finitely presented ring.
    \item An affine scheme of finite presentation over $k$ is explicitly given if it is the spectrum of an explicitly given $k$-algebra. A morphism between two explicitly given affine schemes of finite presentation over $k$ is explicitly given if it is induced by an explicitly given $k$-algebra homomorphism between their coordinate rings.
    \item A scheme $V$ separated of finite presentation over $k$ is explicitly given if it is written as an explicit finite union of explicitly given open affine schemes $U_i$ with affine overlaps $U_i\cap U_j$, such that the gluing morphisms $U_i\cap U_j\to U_i$ are explicitly given. We shall call $(U_i)$ an effective presentation of $V$.
    \item Let $V$ and $W$ be explicitly given schemes of finite presentation over $k$, with effective presentations $(U_i)_{i\in I}$ and $(T_{j})_{j\in J}$ respectively. A morphism $f:V\to W$ over $k$ is explicitly given if there is another effective presentation $(U_i')_{i\in I'}$ of $V$ and an explicit function $j:I'\to J$ such that $f(U_i')\subseteq T_{j(i)}$ for every index $i$, and the following hold:
    \begin{enumerate}
        \item $f|_{U_i'}:U_i'\to T_{j(i)}$ is an explicitly given morphism for all $i\in I$, and
        \item the inclusions of $U_i'\cap U_k$ into $U_i'$ and $U_k'$ are explicitly given morphisms for all $(i,k)\in I'\times I$.
    \end{enumerate}
\end{enumerate}
Finally, a finite system $S$ of algebraic varieties over a field $K$ is explicitly given if there is an explicitly given finitely presented domain $k\subset K$ and a finite system $S_k$ of schemes over $k$, whose vertices are explicitly given separated schemes of finite type over $k$ and whose arrows are explicitly given morphisms between those schemes, such that $S$ is obtained from $S_k$ by base change.

\subsection{Decidability of finiteness} Here, we shall prove Theorem \ref{mainthm0}. Let us begin with a lemma.

\begin{lemma}
\label{finite-order}
Let $X$ be an integral separated scheme. If $f$ is a dominant endomorphism of $X$ of finite order, then $f$ is an automorphism.
\end{lemma}

\begin{proof}
Let $k(X)$ be the function field of $X$. Since $f$ is dominant, it induces an inclusion $f^*:k(X)\to k(X)$. Since there exist $0\leq m<n$ such that $f^m=f^n$, in fact $f^*$ must induce an automorphism of $k(X)$, and $(f^{n-m})^*$ is the identity on $k(X)$. This implies that $f^{n-m}$ is the identity on a dense open subscheme of $X$. Since $X$ is separated, this implies that $f^{n-m}=\id_X$.
\end{proof}

\begin{proposition}
\label{dominant}
Let $k$ be a finitely presented domain of characteristic zero. Let $S$ be a finite path-connected system of dominant morphisms between integral separated schemes of finite type over $k$. If $\langle S\rangle$ is finite, then $|\langle S\rangle|\leq C(S_0)$ where $C(S_0)$ is a constant that only depends on $S_0$.
\end{proposition}

\begin{proof}
Let $S$ be a finite system as in the statement of the proposition. By Lemma \ref{finite-order} and our assumptions on $S$, every endomorphism of an object in $\langle S\rangle$ is an automorphism of finite order. It follows that $\langle S\rangle$ is a groupoid. Fix $Z\in S_0$. For every $W\in S_0$, the set $\Hom_{\langle S\rangle}(Z,W)$ is a torsor under the group $\langle S\rangle_Z$, so we have
$|\langle S\rangle|=|\langle S\rangle_Z|^{|S_0|}$.
Now, $\langle S\rangle_Z$ is a finite subgroup of $\Aut(Z/k)$. Fix two closed points $x$ and $x'$ of the integral scheme $Z$ such that the characteristics of the residue fields $\kappa(x)$ and $\kappa(x')$ are coprime. Setting $X=\{x\}$ and $X'=\{x'\}$, we see by Proposition \ref{prop1} that the composition of group homomorphisms
$$\langle S\rangle_Z\to \Aut(Z/k)\to\Aut Z(A_X/J_X^2)\times \Aut(A_X/J_X^2)$$
is injective. Since the right hand side only depends on $Z$, we are done.
\end{proof}

\begin{thm}{\ref{mainthm0}}
Let $k$ be a field of characteristic zero. There exists an algorithm to determine, given an explicit finite system of irreducible algebraic varieties over $k$, whether or not the category it generates is finite.
\end{thm}

\begin{proof}
Let $S$ be a finite system of irreducible algebraic varieties over $\bar\Q$. By spreading out, we can explicitly determine a finitely generated subring of $\bar\Q$ such that $S$ descends to a finite system of integral separated schemes flat of finite type over $k$. We shall proceed by induction on $|S_1|$ and $\max_{V\in S_0}\dim V$. By Lemma \ref{connected}, we may assume that $S$ is path-connected. If every $f\in S_1$ is dominant, then by Proposition \ref{dominant} it is decidable whether or not $\langle S\rangle$ is finite. So assume that some $f\in S_1$ is not dominant, and write $A,B\in S_0$ for the source and target of $f$ respectively. Let $Z$ be the Zariski closure of the image of $f$ in $B$, equipped with the reduced closed subscheme structure. Note that $Z$ is integral, separated, and flat of finite type over $k$. Consider the finite systems $S'$ and $S''$ where:
\begin{itemize}
    \item[(a)] $S_0'=S_0$ and $S_1'=S_1\setminus\{f\}$, and
    \item[(b)] $S_0''=\{Z\}$ and $S_1''=\{(f\circ g)|_Z:g\in\Hom_{\langle S'\rangle }(B,A)\}$.
\end{itemize}
Since $|S_1'|<|S_1|$, by inductive hypothesis it is decidable whether or not $\langle S'\rangle$ is finite, so we may assume it is. This implies that $|S_1''|$ is finite, and since $Z$ has smaller dimension than $B$, by inductive hypothesis it is decidable whether or not $\langle S''\rangle$ is finite. We may assume $\langle S''\rangle$ is finite. Then for every $V,W\in S_0$, if $g\in\Hom_{\langle S\rangle}(V,W)$ then either $g\in\Hom_{\langle S'\rangle}(V,W)$ or
$$g=h_2\circ e\circ f\circ h_1$$
for some $h_1\in\Hom_{\langle S'\rangle}(V,A)$, $e\in \langle S''\rangle_1$, and $h_2\in\Hom_{\langle S'\rangle}(B,W)$. It follows that $\langle S\rangle$ is finite. This completes the induction and the proof.
\end{proof}

\subsection{Burnside finiteness criterion} We now prove the following extended version of Theorem \ref{mainthm1}, by adapting our proof of Theorem \ref{mainthm0}.

\begin{theorem}
\label{extmainthm1}
Let $k$ be a field. Let $C$ be a finitely generated category of irreducible algebraic varieties over $k$. Then $C$ is finite if and only if:
\begin{enumerate}
    \item $C$ is $n$-torsion for some $n\geq1$, or
    \item $C$ is torsion and $k$ has characteristic zero.
\end{enumerate}
\end{theorem}

\begin{proof}
The ``only if'' direction is clear, so we shall focus on the ``if'' direction. Let $S$ be a finite system generating $C$. As in the proof of Theorem \ref{mainthm0}, we proceed by induction on $|S_1|$ and $\max_{V\in S_0}\dim V$. By Lemma \ref{connected}, we may assume that $S$ is path-connected. First, suppose that every $f\in S_1$ is dominant. Since $C$ is torsion, arguing as in the proof of Proposition \ref{dominant}, we see that $C$ is a finitely generated groupoid, and hence its finiteness is equivalent to the finiteness of the finitely generated group $C_V$ for some (or any) $V\in C_0$. Now, we have the following.
\begin{enumerate}
    \item If $C$ is $n$-torsion for some $n\geq1$, then since the group $C_V$ is residually finite by Theorem \ref{bl0}(1), it is finite by the resolution of the restricted Burnside problem (Theorem \ref{zelmanov}).
    \item If $C$ is torsion and $k$ has characteristic zero, then since $C_V$ is virtually torsionfree by Theorem \ref{bl0} it is finite.
\end{enumerate}
This proves the theorem in the case where every $f\in S_1$ is dominant. If some $f\in S_1$ is not dominant, then we argue as in the proof of Theorem \ref{mainthm0} to construct systems $S'$ and $S''$ with smaller complexity, such that finiteness of $C'=\langle S'\rangle$ and $C''=\langle S''\rangle$ imply the finiteness of $C$. If $C$ satisfies condition (1) or (2) of the theorem, then clearly $C'$ and $C''$ also satisfy the same condition. Thus, $C'$ and $C''$ are finite by inductive hypothesis, and hence $C$ is finite.
\end{proof}

\section{Dynamical corollaries}\label{sect:4}
We record some dynamical corollaries of Theorem \ref{extmainthm1}, by transferring properties of monoids to their orbits. The following is a refined version of Corollary \ref{cor1}.

\begin{corollary}
\label{orbit}
Let $k$ be a field. Let $M$ be a finitely generated monoid acting on an algebraic variety $V/k$. Let $x\in V(k)$. The following are equivalent.
\begin{enumerate}
    \item $M\cdot x$ is finite.
    \item There exists $n\geq1$ with $|N\cdot x|\leq n$ for every $2$-generated submonoid $N\leq M$.
    \item There exists $n\geq1$ with $|\langle f\rangle\cdot g(x)|\leq n$ for every $f,g\in M$.
\end{enumerate}
\end{corollary}

\begin{proof}
    It is clear that $(1)\implies (2)\implies (3)$, so it remains to show $(3)\implies (1)$. Suppose $(3)$ holds. Let $Z$ be the Zariski closure of $M\cdot x$ in $V$, equipped with the reduced closed subscheme structure. Let us now construct a finite system $S$ of irreducible varieties, where the elements of $S_0$ are the irreducible components of $Z$ and the elements of $S_1$ are the restrictions of the elements of a finite generating set of $M$ to the irreducible components of $Z$. Let $C=\langle S\rangle$. If $f\in C_1$ is any endomorphism of an object $A$ in $C$, we claim that $f$ has finite order $\leq n$, i.e., $C$ is $n$-torsion. Indeed, there exists some $0\leq m<n$ such that $A(k)$ contains a Zariski dense set $T$ of points satisfying $f^n(y)=f^m(y)$ for all $y\in T$. Since $A$ is integral and separated, it follows that $f^n=f^m$ on $A$. Thus $C$ is $n$-torsion, and by Theorem \ref{extmainthm1} it follows that $C$ is finite. But then the orbit of $x$ must be finite.
\end{proof}

\begin{thm}{\ref{finorb}}
Let $k$ be a field of characteristic zero.
Let $M$ be a finitely generated monoid acting on an algebraic variety $V/k$. Let $x\in V(k)$. Then the following are equivalent:
\begin{enumerate}
    \item [(a)] $x$ is $M$-periodic.
    \item [(b)] $x$ is $N$-periodic for every $2$-generated submonoid $N\leq M$.
\end{enumerate}
If moreover $M$ is a group, then the above are equivalent to:
\begin{enumerate}
    \item [(c)] $x$ is $\langle f\rangle$-periodic for every $f\in M$.
\end{enumerate}
\end{thm}

\begin{proof}
The implication (a)$\implies$(b) is clear, so let us show (b)$\implies$(a). Assume (b). By \cite[Theorem 1.2]{whang}, there is a constant $n\geq0$, depending only on $V$ and the finitely generated ring over which $V$, $x$, and $M$ are defined, such that $|N\cdot x|\leq n$ for every $2$-generated submonoid $N\leq M$ by $N$-periodicity of $x$. This implies (a) by Corollary \ref{orbit}. Suppose now that $M$ is a group. Since (a)$\implies$(c) is clear, we shall show (c)$\implies$(a). Let us assume (c). In light of Corollary \ref{orbit}, it suffices to show that there is a constant $n\geq0$ such that $|\langle f\rangle \cdot g(x)|\leq n$ for every $f,g\in M$, or (since $M$ is a group) equivalently $|\langle h^{-1}f h\rangle\cdot x|\leq n$ for every $f,h\in M$, or equivalently $|\langle g\rangle\cdot x|\leq n$ for every $g\in M$. Now, since $\langle g\rangle\cdot x$ is finite for every $g\in M$ by hypothesis, the desired result follows again by Theorem \cite[Theorem 1.2]{whang}.
\end{proof}


\begin{thebibliography}{1}

\bibitem{bl}
Bass, H., Lubotzky, A.
\emph{Automorphisms of groups and of schemes of finite type.}
Israel Journal of Mathematics volume 44, pages 1-22 (1983)

\bibitem{burnside}
Burnside, W.
\emph{On an unsettled question in the theory of discontinuous groups}. 
Q. J. Pure Appl. Math. 33 (1902), no. 2, 230-238.

\bibitem{gs}
Golod, E. S., Shafarevich, I.
\emph{On the class field tower}, Izv. Akad. Nauk. SSSR
Ser. Mat. 28 (1964), no. 2, 261-272.

\bibitem{jacob}
Jacob, G.
\emph{La finitude des représentations linéaires des semi-groupes est décidable,}
Journal of Algebra,
Volume 52, Issue 2,
1978,
Pages 437-459.

\bibitem{ms}
Mandel, A., Simon, I.
\emph{On finite semigroups of matrices.}
Theoretical Computer Science,
Volume 5, Issue 2,
1977

\bibitem{mz}
McNaughton, R., Zalcstein, Y.
\emph{The Burnside problem for semigroups},
Journal of Algebra,
Volume 34, Issue 2,
1975,
Pages 292-299.

\bibitem{an}
Novkov, P., Adjan, S.
\emph{Infinite periodic groups, I} (Russian), Izv. Akad. Nauk
SSSR Ser. Mat. 32 (1968) 212–244; English translation in Math. USSR Izv. 2
(1968), no. 1, 209

\bibitem{reference}
O'Connor, J.J., Robertson, E. F., A history of the Burnside problem, preprint,
http://www-history.mcs.st-andrews.ac.uk/HistTopics/Burnside\_problem.html

\bibitem{schur}
Schur, I. \emph{\:Uber Gruppen linearer Substitutionen mit Koeffizienten aus einem
algebraischen Zahlk\"orper}.
Math. Ann. 71 (1911), no. 3, 355-367.

\bibitem{whang}
Whang, J.P.
\emph{On periodic orbits of polynomial maps.}
arXiv:2305.13529

\bibitem{z1}
Zelmanov, E. I. \emph{Solution of the restricted Burnside problem for groups of odd exponent.} (Russian) Izv. Akad. Nauk SSSR Ser. Mat. 54 (1990), no. 1, 42–59, 221; translation in Math. USSR-Izv. 36 (1991), no. 1, 41–60

\bibitem{z2}
Zelmanov, E. I. \emph{Solution of the restricted Burnside problem for 2-groups.} (Russian) Mat. Sb. 182 (1991), no. 4, 568–592; translation in Math. USSR-Sb. 72 (1992), no. 2, 543–565
\end{thebibliography}
\end{document}